\documentclass[11pt]{amsart}
\newcommand{\R}{\mathbb R}%
\newcommand{\C}{\mathbb C}%
\newcommand{\I}{\mathbb I}

\newtheorem{defn}{Definition}[section]
\newtheorem{thm}[defn]{Theorem}
\newtheorem{lem}[defn]{Lemma}
\newtheorem{prop}[defn]{Proposition}
\newtheorem{cor}[defn]{Corollary}

\newtheorem{rem}[defn]{Remark}

\newcommand{\be}{\begin{equation}}
\newcommand{\ee}{\end{equation}}
\begin{document}
\title[Measure on the space of Lipschitz isometric maps]{A Measure on the space of Lipschitz isometric maps of a compact 1-manifold in $\R^2$}
\author[A. Dasgupta]{Amites Dasgupta}
\address{Statistics and Mathematics Unit, Indian Statistical Institute\\ 203,
B.T. Road, Calcutta 700108, India.\\ e-mail:amites@isical.ac.in}
\author[M. Datta]{Mahuya Datta}
\address{Statistics and Mathematics Unit, Indian Statistical Institute\\ 203,
B.T. Road, Calcutta 700108, India.\\ e-mail:
mahuya@isical.ac.in}
\begin{abstract} Let $M$ be a compact 1-manifold. Given a continuous function $g:M\to \R_+$ we consider the following ordinary differential equation: \begin{center}$\|\dot{f}(t)\|=g(t)$, where $f:M\to \R^2$.\end{center}
We construct a probability measure on the space of almost everywhere differentiable solutions of this differential equation and study this measure. A solution of this equation can be viewed as an isometric immersion of  a compact 1-manifold into $\R^2$. Nash's convergence technique in the proof of isometric $C^1$-immersion theorem plays an important role in the construction.
\end{abstract}
\maketitle
\section{Introduction}

This paper originates in an attempt to understand certain remarks of M. Gromov. In \textit{Partial Differential Relations} (\cite{gromov}), Gromov observes  that most under-determined PDE's that occur in geometry and topology are soft; moreover, the solution spaces of such equations are dense in the space of admissible maps. This is the principal underlying philosophy which brings the geometric theory of PDE in the framework of $h$-principle. In an interview with M. Berger (\cite{berger}), he observes that `there is still much to do: Find a measure for the space of solutions (something like Markov fields) and transform the $h$-principle into measure theory.'  
We do not claim to have understood the very meaning of this comment; however, it gave us an inspiration to work in this direction.

In what follows, $M$ will always denote a 1-dimensional compact manifold; that is, either a closed bounded interval or a circle. We consider a very simple but special differential equation for maps $f:M\to \R^2$, namely the equation $\|\dot{f}(t)\|=g(t)$, where $g:M\to \R_+$ is a given continuous function on $M$. By a solution of this equation we shall mean a continuous, almost everywhere differentiable $f$ which satisfies the equation at all points where the function is differentiable. Our aim here is to find a suitable measure on the solution space so that we can say something statistical about a (random) solution.

This equation has its root in geometry; in fact the solutions can be viewed as maps which induce the metric $g^2\,dt^2$ on $M$. Existence of almost everywhere differentiable (Lipschitz) isometric maps in the general set up is treated in \cite[2.4.11]{gromov}. For the purpose of realizing a random solution in this article, we consider the case of one-dimensional manifolds.


The concept of randomness, in this context, can be compared with Brownian motion in the theory of stochastic processes. Brownian motion can be understood as a map from a probability space $(\Omega,{\mathcal F},P)$ to ${\mathcal C}([0,1],\R)$, where ${\mathcal C}([0,1],\R)$ denotes the space of continuous functions $[0,1]\to\R^2$ with the Borel $\sigma$-algebra under the sup norm. This map can be written as $B(t,\omega)$, $t\in[0,1], \omega\in\Omega$. In our case we need a map from a probability space  $(\Omega,{\mathcal F},P)$ to ${\mathcal C}(M,\R^2)$ which will further satisfy the isometry equation $\|\dot{B}(t,\omega)\|=g(t)$, almost everywhere in $M$. Clearly, the differentiability property and the isometry equation impose further restrictions. The Levy-Ciesielski construction of the Brownian motion provides an analogy for our construction with the above restrictions bringing out the differences.


\section{Reduction of the problem}

Let $g:M\to \R_+$ be any continuous function taking positive real values, where $M$ is a closed and bounded interval $I$ or a round circle $S^1$. By a solution of the differential equation
\begin{equation}\|\dot{f}(t)\|=g(t),\ \ \ \mbox{almost every \ } t\in I\label{pde}\end{equation}
we mean a continuous, almost everywhere differentiable $f:M\to\R^2$ which satisfies the equation at all points where the function is differentiable. Our aim is to define a measure on the solution space ${\mathcal S}$ of equation (\ref{pde}) and to be able to say something statistical about a random solution.
We shall first observe that the problem can be reduced to the case when $g$ is a constant function. 

\begin{prop} Let $g:\I\to\R_+$ and $f:I\to \R^2$ be two continuous functions, where $I=[x_0,x_1]$. Then there exists a diffeomorphism $\phi: [0,1] \to I$ and a constant $c\in\R_+$ such that the following condition is satisfied:
\begin{center}If \ \ $\|\frac{df}{dt}(t)\|= g(t)$  \ on \  $I$ \ then \ $\|\frac{d}{dt}(f\circ\phi)(t)\|=c$ \ on \  $[0,1]$.\end{center}
Also, if $f(x_0)=f(x_1)$  then $(f\circ \phi)(0)=(f\circ\phi)(1)$.\label{reduction}\end{prop}
\begin{proof} Let $\I$ denote the unit interval $[0,1]$. Define $\psi:I\to\I$ by the formula
\begin{center}$\psi(t)=c^{-1}\int_{x_0}^t g(\tau)\,d\tau$ for $t\in I$,\end{center}
where $c$ satisfies the relation $c=\int_{x_0}^{x_1} g(\tau)\,d\tau$. Since $g$ is positive on $I$, $\psi$ is a strictly increasing real valued function on $I$. Moreover, $\psi(x_0)=0$ and $\psi(x_1)=1$. If we set $\phi=\psi^{-1}:\I\to I$, then $\dot{\phi}(s)=c g(\phi(s))^{-1}$.
Now, if $f$ is differentiable then $\frac{d}{ds}(f\circ \phi)(s) =\dot{\phi}(s)\dot{f}(\phi(s))=c g(\phi(s))^{-1}\dot{f}(\phi(s))$. If $\|\dot{f}(t)\|=g(t)$ for some $t\in I$, then taking norm on both sides of the previous equation we get $\|\frac{d}{dt}(f\circ \phi)(s)\|=c$.
\end{proof}
Let ${\mathcal S}_0$ denote the solution space of the equation
\begin{equation}\|\dot{f}(t)\|=1,\ \ \ \mbox{almost every \ } t\in \I,\label{pde1}\end{equation}
where $\I=[0,1]$. In view of Proposition~\ref{reduction}, it is enough to induce a measure on ${\mathcal S}_0$. Since ${\mathcal S}_0$ is contained in $C(\I,\R^2)$, by the theory of Stochastic processes, we need only to construct a measurable map $F:\I\times \Omega\to \R^2$ for some probability measure space $(\Omega,{\mathcal F},P)$, such that $F(\ ,\omega)\in {\mathcal S}_0$ for every $\omega\in\Omega$ and $F(t,\ )$ is measurable for every $t$. Indeed, such an $F$ induces a measurable map $\widehat{F}:\Omega\to {\mathcal S}_0$ (\cite{billingsley}), and thereby we have the push-forward measure $\widehat{F}_*P$ on $\mathcal S_0$.

Let ${\mathcal S}_0'$ denote the space of all formal solutions of equation (\ref{pde1}). Then ${\mathcal S}_0'$ consists of all $\alpha:\I\to\R^2$ which are Riemann integrable and satisfy $\|\alpha(t)\|=1$ for almost every $t\in\I$.  We shall, in fact, construct a measurable map $G:\I\times\Omega\to\R^2$ such that $G(\ ,\omega)\in {\mathcal S}_0'$. We then define $F$ by the formula $F(t,\omega)=\int_0^tG(s,\omega)\,ds$. To see that $F$ has the desired property observe that the integral operator $\mathcal I$ defined by ${\mathcal I}(\alpha)(t)=\int_0^t\alpha(s)\,ds$ maps ${\mathcal S}_0'$ into ${\mathcal S}_0$.

\section{Construction of a probability measure}

We have already observed in Section 1 that equation (\ref{pde}) can be viewed as a special case of isometry equation for Riemannian manifolds. Here we explain this in further details. Let $(M,g)$ be a Riemannian manifold of dimension $n$ and $q> n$. A map $f:M\to\R^q$ is called an \textit{isometric immersion} if it satisfies on $M$ (locally) the following system of partial differential equations:
\begin{equation}\left\langle\frac{\partial f}{\partial x_i},\frac{\partial f}{\partial x_j}\right\rangle=g_{ij}, \ \ \  0\leq i\leq j\leq n,\label{isometry}\end{equation} where $x_1,x_2,\dots,x_n$ is a local coordinate system on $M$ and $g_{ij}$ are metric coefficients relative to these coordinates. Observe that when $\dim M=1$, the system reduces to equation (\ref{pde}) above. In 1954, Nash had proved the existence of $C^1$ solutions of (\ref{isometry}) under mild restrictions (see \cite{nash}). Starting with a strictly $g$-short immersion $f_0:M\to\R^q$, Nash had constructed a sequence $\{f_k\}$ of \textit{strictly} $g$-\textit{short immersions} (see \cite{nash}) which is Cauchy in the fine $C^1$-topology. An isometric $C^1$ immersion $f:M\to\R^q$ was obtained as the limit of this sequence.

In contrast with Nash's paper, we consider (almost everywhere differentiable) Lipshitz solutions of (\ref{pde}) in the present paper. However, we implicitly use Nash's ideas to construct a measure on the space of Lipschitz isometric maps on a manifold of dimension 1. Interested readers may find an analogue of Nash's result for Lipschitz isometric maps for Riemannian manifolds in \cite{gromov}.

With reference to equation (\ref{pde1}), the shortness condition on $\{f_k\}$ mentioned above is translated into the equation $\|\dot{f}_k\|=\sqrt{c_k}<1$. Therefore, if $\{f_k\}$ converges to some isometric immersion $f$ in the $C^1$-topology then $\{\dot{f}_k\}$  converges to $\dot{f}$ in the fine $C^0$ topology, and in particular, $\|\dot{f}_k\|\to 1$ as $k\to\infty$. In Nash's theorem $\{f_k\}$ was so constructed that $\sum_{k=1}^\infty\|\dot{f}_k-\dot{f}_{k-1}\|<\infty$ so that $\{\dot{f}_k\}$ becomes Cauchy in the fine $C^0$ topology. The crucial point was to control the terms $\|\dot{f}_k-\dot{f}_{k-1}\|$, $k=1,2,\dots$ as one constructs the sequence. Now if one looks at this geometrically, $\dot{f}_k$ lies on the circle of radius $\sqrt{c_k}$ for each $k$. Without loss of generality we may assume that the sequence $\{c_k\}$ increases monotonically to 1. For any $t$, plot the vector $\dot{f}_k(t)$ in $\R^2$ and consider the chord through $\dot{f}_k(t)$ which is orthogonal to the vector $\dot{f}_k(t)$. This chord intersects the circle of radius $\sqrt{c_{k+1}}$ at two points, say $x_1$ and $x_2$. Consider the shorter arc $C$ of the circle  which is intercepted between these two points. Then for any point $x$ on this arc, the distance between $x$ and $\dot{f}_k(t)$ is bounded by $\sqrt{c_{k+1}-c_k}$. Therefore, if we choose $f_{k+1}$ so that $\dot{f}_{k+1}(t)$ lies on $C$, then the uniform convergence of the series $\sum_{k=1}^\infty\sqrt{c_{k+1}-c_k}$ ensures the $C^0$ convergence of $\{\dot{f}_k\}$.
The construction of the sequence and its convergence can be realised as a random walk in the space of strictly short immersions which converges to a random solution of (\ref{pde1}).

\begin{lem} There exists a sequence of real numbers $\{c_k\}$ converging to $1$ such that
$\sum_{k=1}^\infty\sqrt{c_k-c_{k-1}}<\infty$.
\label{approximation}
\end{lem}
\begin{proof} Fix a real number $\alpha$ such that $0<\alpha<1$. Then starting with any $c_0$, $0\leq c_0<1$, define $c_n=c_{n-1}+\alpha(1-c_{n-1})$ for $n\geq 1$. Then $\{c_n\}$ has the desired properties.
\end{proof}

We now fix some notation.\\

\noindent\textbf{Notation 1.} For any positive integer $k$, $r_k:[0,1]\to \{-1,1\}$ denote the Rademacher function defined by
\begin{center}$r_k(t)= 1-2($the $k$-th dyadic coefficient of $t)$\end{center}
for all $t\in [0,1]$. Clearly, $r_k$ is a step function having jumps at $j/2^k$ for $j=1,2,\dots ,2^k-1$.\\

\noindent\textbf{Notation 2.} For every positive integer $k$, and $j\in\{1,2,\dots 2^{k-1}\}$, we define an interval $I_{kj}$ as follows:
$$\begin{array}{lcl}I_{kj} & = & \left\{
\begin{array}{rcl}
\left[\frac{j-1}{2^{k-1}},\frac{j}{2^{k-1}}\right) & \mbox{if} & 1 \leq j < 2^{k-1}\\
\mbox{}& & \\
\left[\frac{2^{k-1}-1}{2^{k-1}},1\right] & \mbox{if} & j=2^{k-1}
\end{array}
\right.\end{array}$$

\noindent\textbf{Notation 3.} Let $\Omega_0=\{-1,1\}$ be the probability measure space, where $-1$ and $1$ both occur with equal probability. Let $\Omega$ denote the product measure space $\Omega=\Pi_{k=1}^\infty\Pi_{j=1}^{2^{k-1}}\{-1,1\}$. \\

To simplify the presentation, we have identified $\R^2$ with $\C$ in the subsequent discussion.

\begin{lem} Given an increasing sequence of real numbers, $\{c_n\}$, we can generate a random walk in the space of all piecewise continuous paths $\alpha(t)$ such that $\|\alpha(t)\|<1$ for all $t$. The $n$-step truncation $\alpha_n:\I\times\Omega\to \C$ of the random walk is given by
$$\alpha_n(t,\{d_{kj}\})=\sqrt{\frac{c_n}{c_0}}\,\alpha_0\exp i\left\{\sum_{k=1}^{n}\sum_{j=1}^{2^{k-1}} r_k(t)1_{I_{kj}}(t) d_{kj}\sin^{-1}\sqrt{\frac{c_k-c_{k-1}}{c_k}} \right\},$$   
when $c_0\neq 0$, where $\|\alpha_0\|=c_0$; for $c_0=0$ it is given by
$$\alpha_n(t,\{d_{kj}\})=\sqrt{c_n}\exp i\left\{\sum_{k=1}^{n}\sum_{j=1}^{2^{k-1}} r_k(t)1_{I_{kj}}(t) d_{kj}\sin^{-1}\sqrt{\frac{c_k-c_{k-1}}{c_k}} \right\},$$
where $\{d_{kj}\}_{j=1,k=1}^{2^{k-1},\infty}$ is a collection of independent and identically distributed (in short, iid) random variables taking values $\pm 1$ with equal probaility defined on the product probability space $(\Omega,{\mathcal F},P)$ and $1_{I_{kj}}$ denote the characteristic function of the inerval $I_{kj}$.
\label{random walk}\end{lem}
\begin{proof}
Let $\alpha_0:\I\to\C$ be a constant path such that $\|\alpha_0\|=\sqrt{c_0}$.
We want to construct a sequence of step functions $\alpha_n:\I\to\R^2$ such that $\|\alpha_n\|=\sqrt{c_n}$. This is done inductively in the following way. Suppose that we have constructed $\alpha_{n}$ which is constant on some maximal interval $[x_0,x_1)$. Consider the chord of the circle of radius $\sqrt{c_{n+1}}$ through the point $\alpha_n(x_0)$. If the endpoints of the chords are respectively $e_1$ and $e_2$ then we define $\alpha_{n+1}$ on $[x_0,x_1)$ as follows: $\alpha_{n+1}$ takes constant values $e_1$ and $e_2$ respectively on the intervals $[x_0,(x_0+x_1)/2)$ and $[(x_0+x_1)/2,x_1)$. Observe that there are two possible solutions for $\alpha_{n+1}$. If we iterate this process indefinitely then it generates a random walk in the space of simple functions.

Explicitly, we define $\alpha_1$ on $\I$ as follows: If $c_0\neq 0$ then

\begin{center}$\begin{array}{rcl}\alpha_1(t,\{d_{kj}\}) & = & \left\{
\begin{array}{ll}\sqrt{\frac{c_1}{c_0}}\alpha_0\exp \left\{id_{11}\sin^{-1}\sqrt{\frac{c_1-c_0}{c_1}}\,\right\} & \mbox{ if } 0\leq t<1/2\\
\sqrt{\frac{c_1}{c_0}}\alpha_0\exp \left\{-id_{11}\sin^{-1}\sqrt{\frac{c_1-c_0}{c_1}}\,\right\} & \mbox{ if } 1/2\leq t\leq 1\end{array}\right.\end{array}$\end{center}
The function $\alpha_1$ takes constant value in each of the open dyadic subinterval, $\|\alpha_1(t)\|=\sqrt{c_1}$ for all $t$ and the convex hull of the image of $\alpha_1$ contains $\alpha_0(t)$ for all $t$. The randomness is introduced through $d_{11}$ which takes the value $\pm 1$ with equal probability.

Using the Rademacher function $r_1$ (see Notation 1), the expression for $\alpha_1$, when $c_0\neq 0$, can be written as
\begin{center}$\begin{array}{rcl}\alpha_1(t,\{d_{kj}\}) & = &
\sqrt{\frac{c_1}{c_0}}\alpha_0\exp \left\{id_{11}r_1(t)\sin^{-1}\sqrt{\frac{c_1-c_0}{c_1}}\,\right\},\ \ 0\leq t\leq 1\end{array}$\end{center}

If $c_0=0$ then define
\begin{center}$\begin{array}{rcl}\alpha_1(t,\{d_{kj}\}) & = &
\sqrt{c_1}\exp \left\{id_{11}r_1(t)\pi/2\right\}, \ \ 0\leq t\leq 1\end{array}$\end{center}

On each of the intervals $[0,1/2)$ and $[1/2,1]$, $\alpha_1$ is constant. Therefore, after restricting $\alpha_1$ to any of these subintervals we can perform the same process to obtain $\alpha_2$. For this, we need to introduce randomness for each of these subintervals. This gives us the following expression for $\alpha_2$:

\begin{center}$\begin{array}{rcl}\alpha_2(t,\{d_{kj}\}) & = & \left\{
\begin{array}{ll}\sqrt{\frac{c_2}{c_1}}\,\alpha_1\exp \left\{id_{21}\sin^{-1}\sqrt{\frac{c_2-c_1}{c_2}}\,\right\} & \mbox{ if } 0\leq t<1/2^2\\
\sqrt{\frac{c_2}{c_1}}\,\alpha_1\exp \left\{-id_{21}\sin^{-1}\sqrt{\frac{c_2-c_1}{c_2}}\,\right\} & \mbox{ if } 1/2^2\leq t< 1/2\end{array}\right.\\
 & = &\sqrt{\frac{c_2}{c_1}}\,\alpha_1\exp \left\{ir_2(t)d_{21}\sin^{-1}\sqrt{\frac{c_2-c_1}{c_2}}\,\right\} \ \ \ \mbox{ if } 0\leq t<1/2\end{array}$\end{center}

\begin{center}$\begin{array}{rcl}\alpha_2(t,\{d_{kj}\}) & = & \left\{
\begin{array}{ll}\sqrt{\frac{c_2}{c_1}}\,\alpha_1\exp \left\{id_{22}\sin^{-1}\sqrt{\frac{c_2-c_1}{c_2}}\,\right\} & \mbox{ if } 1/2\leq t< 3/4\\
\sqrt{\frac{c_2}{c_1}}\,\alpha_1\exp \left\{-id_{22}\sin^{-1}\sqrt{\frac{c_2-c_1}{c_2}}\,\right\} & \mbox{ if } 3/4\leq t\leq 1\end{array}\right.\\
& = & \sqrt{\frac{c_2}{c_1}}\,\alpha_1\exp \left\{ir_2(t)d_{22}\sin^{-1}\sqrt{\frac{c_2-c_1}{c_2}}\,\right\} \ \ \ \mbox{ if } 1/2\leq t\leq 1,\end{array}$\end{center}
where $\alpha_1=\alpha_1(t,\{d_{kj})\}$ is the random function defined in the previous step.

The above expressions can be combined into one equation as follows for the case $c_0\neq 0$:
\begin{center}$\begin{array}{rcl}\alpha_2(t,\{d_{kj}\}) & = & \sqrt{\frac{c_2}{c_1}}\,\alpha_1\exp \left\{\sum_{j=1}^2 ir_2(t)1_{I_{2j}}d_{2j}\sin^{-1}\sqrt{\frac{c_2-c_1}{c_2}}\,\right\}\\
& = & \sqrt{\frac{c_2}{c_0}}\,\alpha_0\exp \left\{\sum_{k=1}^2\sum_{j=1}^{2^{k-1}} ir_k(t)1_{I_{kj}}(t)d_{kj}\sin^{-1}\sqrt{\frac{c_k-c_{k-1}}{c_k}} \,\right\},\end{array}$\end{center}
where $I_{kj}=[(j-1)/2^{k-1},j/2^{k-1})$, $j=1,2,\dots,2^{k-1}$. Similarly, we can also get an expression for $\alpha_2$ in the case $c_0=0$. We can apply induction argument to complete the proof of the lemma.\end{proof}


\textbf{Observation:} The random function $\alpha_n$ constructed in the above lemma has the following properties: For each $\omega\in\Omega$,
\begin{enumerate}\item $\|\alpha_n(t,\omega)\|=\sqrt{c_n}$ for all $t\in\I$;
\item $\alpha_n(\ ,\omega)$ is constant on each interval $\left[\frac{j-1}{2^n},\frac{j}{2^n}\right)$, $1\leq j\leq 2^{n}$; 
\item $\int_0^1\alpha_n(t,\omega)\, dt=\alpha_0$ for each $n$. This is clearly true for $n=1$; moreover, for any $n>1$  we can write 
\begin{center}$\int_0^1\alpha_n(t,\omega)\, dt=
\sum_{j=1}^{2^{n-1}}\int_{\frac{j-1}{2^{n-1}}}^{\frac{j}{2^{n-1}}}\alpha_n(t,\omega)\, dt= \frac{1}{2^{n-1}}\sum_{j=1}^{2^{n-1}}\alpha_{n-1}(\frac{j-1}{2^{n-1}},\omega)$\end{center} 
which is further equal to $\int_0^1\alpha_{n-1}(t,\omega)\, dt$; 
\item For each $t\in\I$, $\omega\mapsto\alpha_n(t,\omega)$ is a measurable function on $\Omega$.\end{enumerate}

\begin{lem} For any infinite sequence $\{c_n\}$ characterised by Lemma~\ref{approximation}, the following series
$$\sum_{k=1}^\infty\sum_{j=1}^{2^{k-1}} r_k(t)1_{I_{kj}}(t)d_{kj}\sin^{-1}\sqrt{\frac{c_k-c_{k-1}}{c_k}}$$
converges uniformly for each random sequence $\{d_{kj}\}$, where $d_{kj}\in\{1,-1\}$.\label{convergence}
\end{lem}
\begin{proof} It is enough to prove that the series converges absolutely. First observe that for a given $t$ and $k$, there exists a unique $j$ such that $t\in I_{kj}$. Hence, we need only to prove that the series $\sum_{k=1}^\infty\sin^{-1}\sqrt{\frac{c_k-c_{k-1}}{c_k}}$ is convergent. Since the sequence $c_k$ is convergent, $\sqrt\frac{c_k-c_{k-1}}{c_k}$ converges to zero. Therefore, the two series, namely $\sum_{k=1}^\infty\sin^{-1}\sqrt{\frac{c_k-c_{k-1}}{c_k}}$ and $\sum_{k=1}^\infty\sqrt{\frac{c_k-c_{k-1}}{c_k}}$ will converge or diverge simultaneously. This follows from the fact that $\lim_{x\to 0}\frac{\sin x}{x}=1$. Since by  hypothesis, the sequence $\{c_k\}$ converges to non-zero real number and $\sum_{k=1}^\infty\sqrt{c_k-c_{k-1}}<\infty$, we conclude that
$\sum_{k=1}^\infty\sqrt{\frac{c_k-c_{k-1}}{c_k}}$ also converges. This proves the lemma.\end{proof}

Let $\{\alpha_n\}$ be as in Lemma~\ref{random walk}. Define $\alpha:\I\times\Omega\to \R^2$ by
$$\begin{array}{rcl}\alpha(t,\{d_{kj}\}_{j=1,k=1}^{2^{k-1},\infty}) & = & \lim_{n\to\infty}\alpha_n(t,\{d_{kj}\}_{j=1,k=1}^{2^{k-1},n})\end{array}.$$
It follows from Lemma~\ref{convergence} that the limit exists for all $t$. Further, the map $\alpha$ so defined is measurable since it is the limit of a sequence of measurable functions  $\alpha_n$.
We summarise properties of $\alpha$ in the following theorem.
\begin{thm} Let $\{c_n\}$ be an infinte sequence of real numbers such that $\lim_{n\to\infty}c_n=1$ and  $\sum_{n=1}^\infty\sqrt{c_n-c_{n-1}}$ is finite. Then the random walk generated by $\{c_n\}$ as described in Lemma~\ref{random walk} converges to some measurable function $\alpha:\I\times\Omega\to \R^2$, such that for each $\omega$,
\begin{enumerate}\item[$(i)$] $\alpha(\ ,\omega)$ is continuous at all $t\in \I\setminus D$, where $D$ is the set of all dyadic rationals. In particular, $\alpha(\ ,\omega)$ is Riemann integrable on $[0,1]$.
\item[$(ii)$]  $\alpha(\ ,\omega)$ is right continuous with left limits, at all dyadic rationals.
\item[$(iii)$] $\|\alpha(t,\omega)\|=1$ for all $t\in \I$ and $\int_0^1\alpha(t,\omega)\,dt=\alpha_0$.\end{enumerate}\label{alpha theorem}
\end{thm}
\begin{proof}Indeed, if $x\in\I\setminus D$, then each $\alpha_n(\ ,\omega)$ is continuous at $x$ (Observation (2)), so that $\lim_{t\to x}\alpha_n(t,\omega)=\alpha_n(x,\omega)$. Now, since $\alpha_n(\ ,\omega)$ converges uniformly to $\alpha(\ ,\omega)$, we have
\begin{center}$\lim_{t\to x}\lim_{n\to\infty}\alpha_n(t,\omega)=\lim_{n\to\infty}\lim_{t\to x}\alpha_n(t,\omega)$\end{center}
which implies that $\lim_{t\to x}\alpha(t,\omega)=\alpha(x,\omega)$, that is, $\alpha(\ ,\omega)$ is continuous at $x$. This proves (i). To prove (ii), note that $\alpha_n(\ ,\omega)$ is right-continuous with left limits at each dyadic rational and that the sequence $\alpha_n(\ ,\omega)$ converges uniformly to $\alpha(\ ,\omega)$. Finally, (iii) follows from observations (1) and (3) listed above, as $\alpha_n$ converges to $\alpha$.
\end{proof}

\noindent\textbf{Observation:} Let ${\mathcal D}[0,1]$ denote the set of all functions $[0,1]\to\R^2$ which are right continuous with left limits. We endow it with the Skorohod topology (\cite{billingsley}) and consider the associated Borel $\sigma$-field on it. It follows from Theorem~\ref{alpha theorem} that $\alpha$ induces a map $\alpha_*:\Omega\to {\mathcal D}[0,1]$ defined by $\alpha_*(\omega)(t)=\alpha(t,\omega)$ for $\omega\in\Omega$ and $t\in \I$. It follows from Observation (4) that $\alpha_t:\Omega\to\R^2$ ($t\in\I$), defined by $\alpha_t(\omega)=\alpha(t,\omega)$, is the limit of a sequence of measurable functions; hence $\alpha_t$ is measurable for each $t\in\I$. This makes $\alpha_*$ measurable (see \cite{billingsley}). Further observe that $\alpha_*$ actually takes values in ${\mathcal S}'_0$ (defined as in Section 2). Hence, $\alpha_*$ induces a measure, namely the push-forward measure $\alpha_*P$, on ${\mathcal S}_0'$.

We are now in a position to prove the main result of this paper.

\begin{thm} Let $a,b$ be two points in $\R^2$ such that $\|b-a\|<1$. There is a probability measure $\mu$ on the space of all continuous paths $f:\I\to\R^2$ which satisfy the relation $\|\dot{f}(t)\|=1$ for almost all $t\in [0,1]$ and $f(0)=a$ and $f(1)=b$.
\label{main}
\end{thm}
\begin{proof} Define $f_0:\I\to\R^2$ by $f_0(t)=a+t(b-a)$ for $t\in\I$, so that $\dot{f}_0(t)=b-a$. The condition $\|b-a\|<1$ implies that $f_0$ is strictly short relative to the metric $dt^2$ on $\I$, and hence the space of isometric paths is non-empty by Nash's theorem (\cite{nash}). Let $c_0=\|b-a\|$ and extend it to a sequence $\{c_n\}$ satisfying conditions of Lemma~\ref{approximation}. Then taking $\alpha_0=\dot{f}_0$ we can construct a sequence $\{\alpha_n\}$ as in Lemma~\ref{random walk} and get $\alpha$ as the limit of this sequence. Define a function $f:[0,1]\times\Omega\to\R^2$ as follows:
\begin{center}$\begin{array}{rcl}
f(t,\omega) & =: & a\,+\,\int_0^t\alpha(s,\omega)\,ds\end{array}$\end{center}
By the Fundamental Theorem of Calculus, $f(\ ,\omega)$ is absolutely continuous for each $\omega$ and $\dot{f}(t,\omega)=\alpha(t,\omega)$ for all $t\in \I\setminus D$ since $\alpha$ is continuous at these points. Further, since we can express $f$ as the limit of Riemann sums, it follows from the measurability of each $\alpha_t$ that $f_t:\omega\mapsto f(t,\omega)$ is also measurable for each $t\in\I$. Thus $f_*:\omega\mapsto f(\ ,\omega)$ is a measurable map $\Omega\to {\mathcal C}([0,1],\R^2)$ (see \cite{billingsley}), where ${\mathcal C}([0,1],\R^2)$ is the space of continuous functions from $[0,1]$ to $\R^2$ with the sup topology. In particular, $f_*$ takes values in ${\mathcal S}_0$ since $\dot{f}(t,\omega)=\alpha(t,\omega)$ for all $t\in \I\setminus D$ and $\|\alpha(t,\omega)\|=1$. Further, since $\int_0^1\alpha(t,\omega)\,dt=\alpha_0$, it follows that $f(\ ,\omega)$ has the same endpoints as $f_0$ for each $\omega$. The pushforward measure $f_*P$ is then the desired measure on the solution space of equation (\ref{pde1}) with the boundary conditions as stated in the theorem.
\end{proof}

As a direct consequence of the Theroem~\ref{main} we obtain
\begin{cor} There is a probability measure on the space of continuous maps $f:S^1\to\R^2$ defined on the unit circle $S^1$ such that $\|\dot{f}(t)\|=1$ for almost all $t\in S^1$ and $f(1)=a$ for some fixed $a$ in $\R^2$.
\end{cor}

Under certain restrictions on the sequence $\{c_k\}$ it is possible to describe the distribution of $\dot{f}(t,\omega)$ in a satisfactory way, as can be seen from the proposition below. We first prove a lemma.

\begin{lem} Let $\theta(t,\{d_{kj}\})$ denote the sum of the series
\begin{center}$\sum_{k=1}^{\infty}\sum_{j=1}^{2^{k-1}} r_k(t) 1_{I_{kj}}(t) d_{kj}\sin^{-1}\sqrt{\frac{c_k-c_{k-1}}{c_k}}$.\end{center} If $\sin^{-1}\sqrt{\frac{c_k-c_{k-1}}{c_k}}=\frac{1}{2^k}$ then for any $t\in\I$, the distribution of $\theta(t,\omega)$ is uniform on $[-1,1]$, where $\omega=\{d_{kj}\}_{j=1,k=1}^{2^{k-1},\infty}$. In particular, for $r,s\in [-1,1]$ with $r<s$
$$P\left\{\omega\in\Omega:\theta(t,\omega)\in(r,s)\right\}=(s-r)/2.$$
\end{lem}
\begin{proof}For a fixed $t$, $\{r_k(t)\}$ is a fixed sequence of $+1$ and $-1$. Hence it is enough to consider the distribution of the sum \begin{center}$\sum_{k=1}^{\infty}\sum_{j=1}^{2^{k-1}} 1_{I_{kj}}(t) d_{kj}\sin^{-1}\sqrt{\frac{c_k-c_{k-1}}{c_k}}$\end{center}
in order to obtain the distribution of $\theta(t,\{d_{kj}\})$. Moreover, for each $k$ there exists exactly one $j$ for which $t\in I_{kj}$. Therefore the distribution of $\theta(t,\{d_{kj}\})$ for a fixed $t$ is the same as the distribution of the series $\sum_{k=1}^{\infty} d_{k}\sin^{-1}\sqrt{\frac{c_k-c_{k-1}}{c_k}}$, where $\{d_k\}$ is a collection of iid random variables taking values $1$ and $-1$ with equal probability defined on a product space $\Pi_{k=1}^\infty\{-1,1\}$. It is a standard result of probability theory that the distribution of $\sum_{k=1}^{\infty} d_{k}\sin^{-1}\sqrt{\frac{c_k-c_{k-1}}{c_k}}$ is uniform on the interval $[-1,1]$ and is the same as the distribution of the Lebesgue measure scaled by the factor $1/2$ (see \cite{kac}).\end{proof}

\begin{rem} {\em Under the restrictions $\sin^{-1}\sqrt{\frac{c_k-c_{k-1}}{c_k}}=\frac{1}{2^k}$, for $k\geq 1$, the random function $\theta(t,\omega)$ has the Markov property; that is, if $0\leq t_0<t_1<\dots<t_{n-1}<t_n\leq 1$ then the distribution of $\{\theta(t_n,\omega)|\theta(t_0)=x_0,\dots,\theta(t_{n-1})=x_{n-1}\}$ is the same as the distribution of $\{\theta(t_n,\omega)|\theta(t_{n-1})=x_{n-1}\}$. Observe that, if $t_1$ and $t_2(>t_1)$ are two real numbers in $[0,1]$ then there exists a unique postive integer $l$ such that $r_i(t_1)=r_i(t_2)$ for all $i\leq l$. Under the given restrictions on $\{c_k\}$,  $\theta(t_1,\{d_{kj}\})=x_1$ uniquely determines the sequence $\{d_{kj}\}$. Let us denote these values by $d_{kj}(t_1,x_1)$. Then, the distribution of $\{\theta(t_2,\omega)|\theta(t_1)=x_1\}$ is uniform in the interval $y+\left[-\frac{1}{2^{l+1}},\frac{1}{2^{l+1}}\right]$, where $y=\sum_{k=1}^{l}\sum_{j=1}^{2^{k-1}} \frac{1}{2^k}r_k(t_1)1_{I_{kj}}(t_1) d_{kj}(t_1,x_1)$.}\end{rem}

\begin{prop} If $\sin^{-1}\sqrt{\frac{c_k-c_{k-1}}{c_k}}=\frac{1}{2^k}$, then for each $t\in \I\setminus D$, the distribution of $\dot{f}(t,\omega)$ is uniform on the arc $C_0:z_0\exp i\theta$, $-1\leq\theta\leq 1$, where $z_0\in S^1$ . Moreover, if $C$ is a subarc of $C_0$ then
$$P\left\{\omega\in\Omega:\dot{f}(t,\omega)\in C\right\}=k\lambda(C)$$ for some constant $k$, where $\lambda$ is the Lebesgue measure on the circle. \end{prop}

\begin{proof} The proof follows from the lemma above, with the observation that $\dot{f}(\ ,\omega)$ exists and is equal to $\alpha(t,\omega)$ for each $t\in\I\setminus D$ and $\omega\in\Omega$. \end{proof}

{\em Acknowldegements}: The second author would like to thank Misha Gromov for his comments on this problem during a discussion in IHES.


\end{document}